\documentclass[3p]{elsarticle}

\usepackage{amsmath,amsthm,amssymb,mathrsfs}
\usepackage{thmtools, thm-restate}
\usepackage{enumerate}
\usepackage{slashed}
\usepackage{txfonts,dsfont}
\usepackage[breaklinks,colorlinks]{hyperref}

\declaretheorem[numberwithin=section]{theorem}
\declaretheorem[sibling=theorem, name=Lemma]{lem}

\declaretheorem[sibling=theorem, name=Remark]{rem}

\numberwithin{equation}{section}

\newcommand{\norm}[1]{\left\lVert#1\right\rVert}%norm%
\newcommand{\abs}[1]{\left\lvert#1\right\rvert}%absolute%
\newcommand{\set}[1]{\left\{#1\right\}}%set%
\newcommand{\hin}[2]{\left\langle#1,#2\right\rangle}%Hermite inner product%

\newcommand*{\To}{\longrightarrow}
\newcommand*{\Rmn}[1]{\uppercase\expandafter{\romannumeral#1}}%uppercase roman number%
\newcommand*{\dif}{\mathop{}\!\mathrm{d}}%differential operator%

\DeclareMathOperator*{\esssup}{ess\,sup}%essential super-mum%

\journal{}

\allowdisplaybreaks

\begin{document}

\begin{frontmatter}

\title{Global existence and convergence of a flow to Kazdan-Warner equation with  non-negative prescribed function\tnoteref{SZ}}

\author[whu1,whu2]{Linlin Sun}
\address[whu1]{School of Mathematics and Statistics, Wuhan University, Wuhan 430072, China}
\address[whu2]{Hubei Key Laboratory of Computational Science, Wuhan University, Wuhan, 430072, China}
\ead{sunll@whu.edu.cn}

\author[mpi]{Jingyong Zhu\corref{zjy}}
\address[mpi]{Max Planck Institute for Mathematics in the Sciences, Inselstrasse 22, 04103 Leipzig, Germany}
\ead{jizhu@mis.mpg.de}

%\fntext[sll]{}

\cortext[zjy]{Corresponding author.}

\tnotetext[SZ]{This research is partially supported by the National Natural Science Foundation of China (Grant Nos. 11971358, 11801420) and the Youth Talent Training Program of Wuhan University. The first author would like thank the Max Planck Institute for Mathematics in the Sciences for good working conditions when this work was carried out.}

\begin{abstract}
We consider an evolution problem associated to the Kazdan-Warner equation on a closed Riemann surface $(\Sigma,g)$
\begin{align*}
    -\Delta_{g}u=8\pi\left(\frac{he^{u}}{\int_{\Sigma}he^{u}{\rm d}\mu_{g}}-\frac{1}{\int_{\Sigma}{\rm d}\mu_{g}}\right)
\end{align*}
where the prescribed function $h\geq0$ and $\max_{\Sigma}h>0$. We prove the global existence and convergence under additional assumptions such as
\begin{align*}
    \Delta_{g}\ln h(p_0)+8\pi-2K(p_0)>0
\end{align*}
for any maximum point $p_0$ of the sum of $2\ln h$ and the regular part of the Green function, where $K$ is the Gaussian curvature of $\Sigma$. In particular, this gives a new proof of the existence result by Yang and Zhu [Proc. Amer. Math. Soc. 145 (2017), no. 9, 3953-3959] which generalizes   existence result of Ding, Jost, Li and Wang [Asian J. Math. 1 (1997), no. 2, 230-248] to the non-negative prescribed function case.
\end{abstract}


\end{frontmatter}

%% main text
\section{Introduction}

Let $\Sigma$ be a closed Riemann surface with a fixed conformal structure. Choose a conformal metric $g$ in the conformal class such that the area of $\Sigma_{g}\coloneqq\left(\Sigma,g\right)$ is one. Let $h$ be a non-negative but nonzero smooth function on $\Sigma$. We consider the following Kazdan-Warner equation
\begin{align}\label{eq:kw8pi}
-\Delta_{g}u=8\pi\left(\dfrac{he^{u}}{\int_{\Sigma}he^{u}\dif\mu_{g}}-1\right).
\end{align}
Here $\Delta_{g}$ is the Laplace–Beltrami operator. The solutions to \eqref{eq:kw8pi} are the critical points of the following functional:
\begin{align*}
J(u)\coloneqq\int_{\Sigma}\left(\dfrac12\abs{\nabla_{g}u}_{g}^2+8\pi u\right)\dif\mu_{g}-8\pi\ln\left(\int_{\Sigma}he^{u}\dif\mu_{g}\right).
\end{align*}

Many mathematicians have contributed to the study of Kazdan-Warner equation. Forty years ago, Kazdan and Warner \cite{kazdan1974curvature} considered the solvability of the equation
\begin{align*}
    -\Delta_{g}u=he^u-\rho,
\end{align*}
where $\rho$ is a constant and $h$ is some smooth prescribed function. When $\rho>0$, the equation above is equivalent to
\begin{align*}
    -\Delta_{g}u=\rho (he^u-1).
\end{align*}
In particular, when $\Sigma_{g}$ is the standard sphere $\mathbb{S}^2$, it is called the Nirenberg problem, which comes from the conformal geometry. It has been studied by Moser \cite{moser1971sharp}, Kazdan and Warner \cite{kazdan1974curvature}, Chen
and Ding \cite{chen1987scalar}, Chang and Yang \cite{chang1987prescribing} and others.

The Kazdan-Warner equation can be also viewed as a special case of the following mean field equation:
\begin{align}\label{mean-field}
    -\Delta_{g}u=\rho\left(\dfrac{fe^{u}}{\int_{\Sigma}fe^{u}\dif\mu_{g}}-1\right),
\end{align}
where $f$ is a smooth function on $\Sigma$. The mean field equation \eqref{mean-field} appears in various context such as the abelian Chern-Simons-Higgs models (see for example \cite{caffarelli1995vortex,tarantello1996multiple,yang2013solitons}). When $f>0$,  the equation \eqref{mean-field} is equivalent to the following equation:
\begin{align}\label{mean-field-type}
      -\Delta_{g}u=\rho\dfrac{e^{u}}{\int_{\Sigma}e^{u}\dif\mu_{g}}-Q,
\end{align}
where $Q\in C^\infty(\Sigma)$ is a given function such that $\int_\Sigma Q\dif\mu_{g}=\rho$. The existence of solutions
of \eqref{mean-field-type} has been widely studied in recent decades. Many partial existence results have been obtained for noncritical cases according to the Euler characteristic
of $\Sigma$ (see for example Brezis and Merle \cite{BreMer91}, Chen and Lin \cite{chen2003topological}, Ding, Jost, Li and Wang \cite{ding1999existence}, Lin \cite{Lin00topological}, Malchiodi  \cite{Mal08morse} and the references therein). Djadli \cite{djadli2008existence}
established the existence of solutions for all surfaces $\Sigma$ when $\rho\neq 8k\pi$ by studying the topology of sublevels to achieve a min-max scheme which already introduced by Djadli
and Malchiodi in \cite{DM2008existence}.

The following evolution problem associated to \eqref{mean-field-type} was also well studied by Cast{\'e}ras  for noncritical cases.
\begin{align}\label{eq:mean-field-type-flow}
\dfrac{\partial e^{u}}{\partial t}=\Delta_{g}u+\rho\dfrac{e^u}{\int_\Sigma e^u\dif\mu_{g}}-Q,\quad u(\cdot,0)=u_0
\end{align}
where $u_0\in C^{2+\alpha}(\Sigma)$. This flow possesses a structure that is very similar to the Calabi and Ricci-Hamilton flows.  When $Q$ is a constant equal to the scalar curvature of $\Sigma$ with respect to the metric $g$, the flow \eqref{eq:mean-field-type-flow} has been studied by Struwe  \cite{struwe2002curvature}. A flow approaching to Nirenberg's problem was studied by Struwe in \cite{Str05flow}. The global existence and convergence of \eqref{eq:mean-field-type-flow} were proved by Cast{\'e}ras in \cite{casteras2015mean}. However, the convergence result there does not include the critical cases, i.e. $\rho=8k\pi$ for $k\in\mathbb{N}$. Recently, when $\rho=8\pi$, a sufficient condition for convergence was given by Li and Zhu in \cite{li2019convergence}. This gives a new proof of the result of Ding, Jost, Li and Wang in \cite{DinJosLiWan97} which was extended by Lin and Chen to  general critical cases \cite{CheLin02sharp} and recently  generalized by Yang and Zhu to  non-negative prescribed function cases in \cite{yang2017remark}.

Motivated by these results, we consider the following evolution problem for \eqref{eq:kw8pi} with non-negative prescribed function:
\begin{align}\label{eq:mean-field-flow}
\dfrac{\partial e^{u}}{\partial t}=\Delta_{g}u+8\pi\left(\dfrac{he^{u}}{\int_{\Sigma}he^{u}\dif\mu_{g}}-1\right),\quad u(\cdot,0)=u_0
\end{align}
where $u_0\in H^{2}(\Sigma)$  and $h$ is a non-negative but nonzero smooth function on $\Sigma$. Since the prescribed function $h$ may be zero on some nonempty subset of $\Sigma$, the global existence and convergence of this flow are subtle. Precisely, we can not use the lower bound of $h$ to do a priori estimates. Therefore, Cast{\'e}ras's proof of global existence for positive prescribed function does not apply to our situation. In addition, the condition (ii) of (1.6) in Cast{\'e}ras's compactness result \cite{casteras2015mean1} actually assumes 
\begin{align*}
    -\frac{\partial e^{u_n}}{\partial t}+\rho e^{u_n}\geq -C, \quad \forall x\in \Sigma, \forall n\geq 1,
\end{align*}
for a sequence of time-slices $u_n\coloneqq u(\cdot,t_n)$. This condition was proved in Proposition 2.1 \cite{casteras2015mean}. However, the proof also need the prescribed function $h$ to be positive. Thus, our a priori estimates in the proof of global existence and blow-up analysis used in the proof of global convergence are both new.

First, we prove the global existence of the flow \eqref{eq:mean-field-flow}.
\begin{theorem}[Global existence]\label{global-existence}
For $u_0\in H^2(\Sigma)$, there is a unique global solution $u\in C^{\infty}\left(\Sigma\times(0,\infty)\right)$ to \eqref{eq:mean-field-flow} with
\begin{align*}
u\in \cap_{0<T<\infty}\left(L^{\infty}\left(0,T;H^2\left(\Sigma\right)\right)\cap H^1\left(0,T;H^1\left(\Sigma\right)\right)\cap H^2\left(0,T;H^{-1}\left(\Sigma\right)\right)\right).
    \end{align*}
    Moreover, for every $0<T<\infty$, there is a positive constant $C\left(T,\norm{u_0}_{H^2\left(\Sigma\right)}\right)$ depending only on $T$, the upper bound of $\norm{u_0}_{H^2\left(\Sigma\right)}$ and $\Sigma_{g}$,
    \begin{align}\label{eq:a-priori}
        \esssup_{0\leq t\leq T}\norm{u(t)}_{H^2\left(\Sigma\right)}+\left(\int_0^T\left(\norm{\dfrac{\partial u(t)}{\partial t}}_{H^1\left(\Sigma\right)}^2+\norm{\dfrac{\partial^2 u(t)}{\partial t^2}}_{H^{-1}\left(\Sigma\right)}^2\right)\dif t\right)^{1/2}\leq C\left(T,\norm{u_0}_{H^2\left(\Sigma\right)}\right),
    \end{align}
    where $u(t)\coloneqq u(\cdot,t)$.
    In particular, if $u_0$ is smooth, then $u$ is smooth. Here the Sobolev spaces $H^{k}(0,T;X)\coloneqq W^{k,2}(0,T;X)$ and  $W^{k,p}(0,T;X)$ consists of all functions $u\in L^p(X\times[0,T])$ such that $\frac{\partial u}{\partial t},\dotsc,\frac{\partial^k}{\partial t^k}$ exists in the weak sense and belongs to $L^p(X\times[0,T])$ and
    \begin{align*}
\norm{u}_{W^{k,p}(0,T;X)}\coloneqq\begin{cases}
\left(\int_0^T\left(\norm{u(t)}_X^p+\sum_{i=1}^k\norm{\frac{\partial t^iu(t)}{\partial t^i}}_{X}^p\right)\dif t\right)^{1/p},& 1\leq p<\infty,\\
\esssup\limits_{0\leq t\leq T}\left(\norm{u(t)}_X+\sum_{i=1}^k\norm{\frac{\partial t^iu(t)}{\partial t^i}}_{X}\right),& p=\infty.
\end{cases}
\end{align*}
\end{theorem}

Then it is interesting to consider the convergence of the flow. To do so, we begin with the monotonicity formula. It gives us that a sequence of positive numbers $t_n\to\infty$ as $n\to\infty$ with
\begin{align*}
    \int_{\Sigma}e^{u_n}\left|\frac{\partial u_n}{\partial t}\right|^2\dif\mu_{g}\to0, \quad \text{as} \quad n\to\infty,
\end{align*}
where $u_n\coloneqq u(t_n)$. If $\|u_n\|_{H^2(\Sigma)}$ is uniformly bounded, then $u_n$ subsequentially converges to a smooth solution of \eqref{eq:kw8pi}. Otherwise, we can get the following lower bound of the functional $J$ along the flow \eqref{eq:mean-field-flow}.

\begin{theorem}\label{thm:lower-bounds}
If the flow \eqref{eq:kw8pi} develops a singularity at the infinity, then we have 
\begin{align*}
    J(u(t))\geq C_0=-4\pi\max_{x\in\Sigma}(A(x)+2\ln h(x))-8\pi\ln\pi-8\pi,\quad\forall t\geq0,
\end{align*}
where $A$ is the regular part of the Green function $G$ which has the following expansion in the normal coordinate system:
\begin{align*}
    G(x,p)=-4\ln r+A(p)+b_1x_1+b_2x_2+c_1x_1^2+2c_2x_1x_2+c_3x_2^2+O(r^3),
\end{align*}
where $r(x)=\mathrm{dist}_g(x,p)$.
\end{theorem}

Last, by imposing certain geometric condition, we get functions whose value under $J$ is strictly less than $C_0$. Consequently, when the flow starts with these functions, the previous $u_n$ will converges in $H^2(\Sigma)$. Moreover, it follows from the {\L}ojasiewicz-Simon  gradient inequality that the convergence of the flow is actually global in time.

\begin{theorem}[Global convergence]\label{thm:converges}
There exists an initial data $u_0\in C^{\infty}(\Sigma)$ such that $u(t)$ converges in $H^2(\Sigma)$ to a smooth solution of \eqref{eq:kw8pi} provided that 
\begin{align}\label{condition-h}
\begin{split}
    &\Delta_{g} h(p_0)+2(b_1(p_0)k_1(p_0)+b_2(p_0)k_2(p_0))\\
    &>-\left(8\pi+b_1^2(p_0)+b_2^2(p_0)-2K(p_0)\right)h(p_0),
\end{split}
\end{align}
where $K$ is the Gaussian curvature of $\Sigma$, $\nabla_{g} h(p_0)=(k_1(p_0),k_2(p_0))$ in the normal coordinate system, $p_0$ is the maximum point of the function $q\mapsto A(q)+2\ln h(q)$.
\end{theorem}

\begin{rem}
As pointed by Ding, Li, Jost and Wang in \cite[Remark 1.1]{DinJosLiWan97}, the inequality \eqref{condition-h} is implied by the following one:
\begin{align*}
    \Delta_{g}\ln{h}(p_0)+8\pi-2K(p_0)>0
\end{align*}
where $p_0$ is the maximum point of the function $q\mapsto A(q)+2\ln h(q)$. 
\end{rem}

\begin{rem}For $\rho\in(0,8\pi)$ and any initial data $u_0\in C^{\infty}\left(\Sigma\right)$, by using a similarly argument, the 
\begin{align*}
    \dfrac{\partial e^{u}}{\partial t}=\Delta_{g}u+\rho\left(\dfrac{he^{u}}{\int_{\Sigma}he^{u}\dif\mu_{g}}-1\right),\quad u(\cdot,0)=u_0
\end{align*}
admits a unique global smooth solution which converges to a solution to 
\begin{align*}
    \Delta_{g}u_{\infty}+\rho\left(\dfrac{he^{u_{\infty}}}{\int_{\Sigma}he^{u_{\infty}}\dif\mu_{g}}-1\right)=0.
\end{align*}
\end{rem}

The remaining part of this paper will be organized as follows. In \autoref{sec:existence}, we prove the global existence of the flow \eqref{eq:mean-field-flow}. In \autoref{sec:blowup}, we prove the number of the singularities is at most one. In \autoref{sec:lower}, we show the lower bound of $J$ along the flow if the singularity occurs. In the last \autoref{sec:convergence}, we prove the global convergence of the flow.

\section{Global existence}
\label{sec:existence}
 The aim of this section is to prove the global existence of the mean field flow \eqref{eq:mean-field-flow}, i.e.  \autoref{global-existence}. 

\begin{proof}[Proof of  \autoref{global-existence}]First, we assume  $u_0\in C^{\infty}\left(\Sigma\right)$. Since the flow is parabolic, the short time existence of \eqref{eq:mean-field-flow} follows from the standard method (e.g. \cite{Fri64partial}). Thus, there exists $T>0$ such that $u\in C^{\infty}(\Sigma\times[0,T])$ is a solution of \eqref{eq:mean-field-flow}.

Along the flow \eqref{eq:mean-field-flow}, it is easy to see
\begin{align}\label{eq:energy}
\dfrac{\dif}{\dif t}\int_{\Sigma}e^{u(t)}\dif\mu_{g}=0
\end{align}
and 
\begin{align}\label{eq:functional}
\dfrac{\dif}{\dif t}J(u(t))=-\int_{\Sigma}e^{u(t)}\abs{\dfrac{\partial u(t)}{\partial t}}^2\dif\mu_{g}.
\end{align}
According to \eqref{eq:functional} and \eqref{eq:energy}, we get
\begin{align}\label{eq:functional-1}
\int_{\Sigma}\left(\dfrac12\abs{\nabla_{g}u(t)}_{g}^2+8\pi u(t)\right)\dif\mu_{g}\leq J(u_0)+8\pi\ln\max_{\Sigma}h+8\pi\ln\int_{\Sigma}e^{u_0}\dif\mu_{g}.
\end{align}
Recall the Trudinger-Moser inequality (cf. \cite[Theorem 1.7]{Fontana93})
\begin{align}\label{eq:TM}
    \ln\int_{\Sigma}e^{u}\dif\mu_{g}\leq\dfrac{1}{16\pi}\int_{\Sigma}\abs{\nabla_{g}u}^2\dif\mu_{g}+\fint_{\Sigma}u\dif\mu_{g}+c,\quad\forall u\in H^1\left(\Sigma\right),
\end{align}
where $c$ is a constant depending only on the Riemann surface $\left(\Sigma,g\right)$. As an immediately consequence of \eqref{eq:TM},
\begin{align}\label{eq:functional-2}
   J(u(t))\geq 8\pi\ln\int_{\Sigma}e^{u(t)}\dif\mu_{g}-8\pi c-8\pi\ln\int_{\Sigma}he^{u(t)}\dif\mu_{g} 
\end{align}
and \eqref{eq:energy} imply that
\begin{align}\label{eq:heu}
    0<C^{-1}\exp\left(-C\norm{u_0}_{H^1\left(\Sigma\right)}^2\right)\leq\int_{\Sigma}he^{u(t)}\dif\mu_{g}\leq C\exp\left(C\norm{u_0}_{H^1\left(\Sigma\right)}^2\right).
\end{align}
Together with
\begin{align*}
\begin{split}
    \dfrac{\dif}{\dif t}\int_{\Sigma}e^{pu(t)}\dif\mu_{g}
    &=p\int_{\Sigma}e^{(p-1)u(t)}\left(\Delta_{g}u(t)+8\pi\dfrac{he^u(t)}{\int_{\Sigma}he^u(t)}-8\pi\right)\dif\mu_{g}\\
    &=-p(p-1)\int_{\Sigma}e^{(p-1)u(t)}\abs{\nabla_g{u(t)}}_g^2\dif\mu_{g}+8p\pi\left(\dfrac{\int_{\Sigma}he^{pu(t)}\dif\mu_{g}}{\int_{\Sigma}he^{u(t)}\dif\mu_{g}}-\int_{\Sigma}e^{(p-1)u(t)}\dif\mu_{g}\right),
    \end{split}
\end{align*}
we have
\begin{align*}
    \dfrac{\dif}{\dif t}\int_{\Sigma}e^{pu(t)}\dif\mu_{g}\leq Cp\exp\left(C\norm{u_0}_{H^1\left(\Sigma\right)}^2\right)\int_{\Sigma}e^{pu(t)}\dif\mu_{g}.
\end{align*}
Thus, 
\begin{align}\label{exponential-integral-bound}
    \int_{\Sigma}e^{pu(t)}\dif\mu_{g}\leq \exp\left[Cp\exp\left(C\norm{u_0}_{H^1\left(\Sigma\right)}^2\right)t\right]\int_{\Sigma}e^{pu_0}\dif\mu_{g}, \quad \forall p\geq1.
\end{align}

In order to get the global existence of solution when $u_0\in H^1(\Sigma)$, it is necessary to derive several a priori estimates \eqref{eq:a-priori}. To do this, we split three steps. 
\begin{description}
\item[Step 1]\begin{quotation}
  $\norm{u(t)}_{H^1\left(\Sigma\right)}\leq C\left(T,\norm{u_0}_{H^1\left(\Sigma\right)}\right)$ for any $t\in[0,T]$.  
\end{quotation} 

Set
\begin{align*}
    A(t)=\set{x\in\Sigma: e^{u(x,t)}\geq\dfrac12\int_{\Sigma}e^{u_0}\dif\mu_{g}.}
\end{align*}
According to \eqref{eq:energy} and \eqref{exponential-integral-bound}, we have
\begin{align*}
    \int_{\Sigma}e^{u_0}\dif\mu_{g}=\int_{\Sigma}e^{u(t)}\dif\mu_{g}=\int_{\Sigma\setminus A(t)}e^{u(t)}\dif\mu_{g}+\int_{A(t)}e^{u(t)}\dif\mu_{g}\leq \dfrac12\int_{\Sigma}e^{u_0}\dif\mu_{g}+C\left(T,\norm{u_0}_{H^1\left(\Sigma\right)}\right)\abs{A(t)}_{g}^{1/2},
\end{align*}
where $\abs{A(t)}_{g}$ stands for the area of $A(t)$. This gives
\begin{align*}
    \abs{A(t)}_{g}\geq C\left(T,\norm{u_0}_{H^1\left(\Sigma\right)}\right)^{-1}>0,\quad\abs{\int_{A(t)}u(t)
    \dif\mu_{g}}\leq C\left(T,\norm{u_0}_{H^1\left(\Sigma\right)}\right)
\end{align*}
and
\begin{align*}
    \abs{\bar u(t)}\coloneqq&\abs{\int_{\Sigma}u(t)
    \dif\mu_{g}}\\
    \leq&\abs{\int_{\Sigma\setminus A(t)}u(t)
    \dif\mu_{g}}+\abs{\int_{A(t)}u(t)
    \dif\mu_{g}}\\
    \leq&\abs{\Sigma\setminus A(t)}^{1/2}\left(\int_{\Sigma\setminus A(t)}u(t)^2
    \dif\mu_{g}\right)^{1/2}+C\left(T,\norm{u_0}_{H^1\left(\Sigma\right)}\right)\\
    \leq&\sqrt{1-C\left(T,\norm{u_0}_{H^1\left(\Sigma\right)}\right)^{-1}}\norm{u(t)}_{L^2\left(\Sigma\right)}+C\left(T,\norm{u_0}_{H^1\left(\Sigma\right)}\right).
\end{align*}
Then, by Poincar\'e inequality, we get
\begin{align*}
    \norm{u(t)}_{L^2\left(\Sigma\right)}\leq c\norm{\nabla_{g}u(t)}_{L^2\left(\Sigma\right)}+\abs{\bar u(t)}\leq c\norm{\nabla_{g}u(t)}_{L^2\left(\Sigma\right)}+\sqrt{1-C\left(T,\norm{u_0}_{H^1\left(\Sigma\right)}\right)^{-1}}\norm{u(t)}_{L^2\left(\Sigma\right)}+C\left(T,\norm{u_0}_{H^1\left(\Sigma\right)}\right),
\end{align*}
which implies the following $L^2$-estimate
\begin{align}\label{L2}
    \norm{u(t)}_{L^2\left(\Sigma\right)}\leq C\left(T,\norm{u_0}_{H^1\left(\Sigma\right)}\right)\left(1+\norm{\nabla_{g}u(t)}_{L^2\left(\Sigma\right)}\right).
\end{align}
Now, applying Young's inequality to \eqref{eq:functional-1}, we obtain
\begin{align*}
    C\geq\int_{\Sigma}\abs{\nabla_{g}u(t)}_{g}^2\dif\mu_{g}-\varepsilon\int_{\Sigma}u(t)^2\dif\mu_{g}-C_{\varepsilon}.
\end{align*}
Choosing small $\varepsilon$ such that
\begin{align*}
    \int_{\Sigma}\abs{\nabla_{g}u(t)}_{g}^2\dif\mu_{g}\leq C\left(T,\norm{u_0}_{H^1\left(\Sigma\right)}\right).
\end{align*}
Thus, together with \eqref{L2}, we can conclude that
\begin{align}\label{eq:H_1}
    \norm{u(t)}_{H^1\left(\Sigma\right)}\leq C\left(T,\norm{u_0}_{H^1\left(\Sigma\right)}\right).
\end{align}

\item[Step 2]
\begin{quotation}
    $\norm{u(t)}_{H^2\left(\Sigma\right)}+\left(\int_0^T\norm{\frac{\partial u(t)}{\partial t}}_{H^1\left(\Sigma\right)}^2\dif t\right)^{1/2}\leq C\left(T,\norm{u_0}_{H^2\left(\Sigma\right)}\right)$ for any $t\in[0,T]$.
\end{quotation}

Set $w(t)=e^{\frac{u(t)}{2}}\frac{\partial u(t)}{\partial t}$. Then
\begin{align*}
    \dfrac12\dfrac{\dif}{\dif t}\int_{\Sigma}\abs{\Delta_{g}u(t)}^2\dif\mu_{g}=&\int_{\Sigma}\Delta_{g}u(t)\Delta_{g}\dfrac{\partial u(t)}{\partial t}\dif\mu_{g}\\
    =&\int_{\Sigma}\left(e^{\frac{u(t)}{2}}w(t)-8\pi\left(\dfrac{he^{u(t)}}{\int_{\Sigma}he^{u(t)\dif\mu_{g}}}-1\right)\right)\Delta_{g}\left(e^{-\frac{u(t)}{2}}w(t)\right)\dif\mu_{g}\\
    =&-\int_{\Sigma}\abs{\nabla_{g}w(t)}_{g}^2\dif\mu_{g}+\dfrac14\int_{\Sigma}w(t)^2\abs{\nabla_{g}u(t)}^2\dif\mu_{g}\\
    &+\dfrac{8\pi}{\int_{\Sigma}he^{u(t)}\dif\mu_{g}}\int_{\Sigma}\hin{e^{\frac{u(t)}{2}}\left(\nabla_{g}h+h\nabla_{g}u(t)\right)}{\nabla_{g}w(t)-\dfrac12w(t)\nabla_{g}u(t)}_{g}\dif\mu_{g}.
\end{align*}
According to \eqref{eq:heu} and \eqref{eq:H_1}, we know that
\begin{align*}
    \norm{u(t)}_{H^1\left(\Sigma\right)}+\dfrac{1}{\int_{\Sigma}he^{u(t)}\dif\mu_{g}}\leq C\left(T,\norm{u_0}_{H^1\left(\Sigma\right)}\right).
\end{align*}
Therefore, Young's inequality implies that
\begin{align*}
    \dfrac{\dif}{\dif t}\int_{\Sigma}\abs{\Delta_{g}u(t)}^2\dif\mu_{g}\leq&-\int_{\Sigma}\abs{\nabla_{g}w(t)}^2_{g}\dif\mu_{g}+\int_{\Sigma}w(t)^2\abs{\nabla_{g}u(t)}^2\dif\mu_{g}+C\left(T,\norm{u_0}_{H^1\left(\Sigma\right)}\right)\left(1+\norm{\nabla_{g}u(t)}_{L^4\left(\Sigma\right)}^2\right).
\end{align*}
Since for all $f\in H^1\left(\Sigma\right)$, we have the following interpolation inequality 
\begin{align}\label{eq:interpolation}
    \norm{f}^2_{L^4\left(\Sigma\right)}\leq c\norm{f}_{L^2\left(\Sigma\right)}\norm{f}_{H^1\left(\Sigma\right)}.
\end{align}
We estimate
\begin{align*}
    \int_{\Sigma}w(t)^2\abs{\nabla_{g}u(t)}^2\dif\mu_{g}\leq& c\norm{w(t)}^2_{L^4\left(\Sigma\right)}\norm{\nabla_{g}u(t)}_{L^4\left(\Sigma\right)}^2\\
    \leq& c\norm{w(t)}_{L^2\left(\Sigma\right)}\norm{w(t)}_{H^1\left(\Sigma\right)}\norm{u(t)}_{H^1\left(\Sigma\right)}\norm{u(t)}_{H^2\left(\Sigma\right)}\\
    \leq&C\left(T,\norm{u_0}_{H^1\left(\Sigma\right)}\right)\norm{w(t)}_{L^2\left(\Sigma\right)}\norm{w(t)}_{H^1\left(\Sigma\right)}\norm{u(t)}_{H^2\left(\Sigma\right)}
\end{align*}
and
\begin{align*}
    \norm{\nabla_{g}u(t)}_{L^4\left(\Sigma\right)}^2\leq c\norm{u}_{H^1\left(\Sigma\right)}\norm{u}_{H^2\left(\Sigma\right)}\leq C\left(T,\norm{u_0}_{H^1\left(\Sigma\right)}\right)\norm{u}_{H^2\left(\Sigma\right)}.
\end{align*}
Hence
\begin{align*}%\label{u'-derivative}
    \dfrac{\dif}{\dif t}\int_{\Sigma}\abs{\Delta_{g}u(t)}^2\dif\mu_{g}\leq&-\dfrac12\int_{\Sigma}\abs{\nabla_{g}w(t)}_{g}^2\dif\mu_{g}+\dfrac12\int_{\Sigma}w(t)^2\dif\mu_{g}+C\left(T,\norm{u_0}_{H^1\left(\Sigma\right)}\right)\norm{w(t)}_{L^2\left(\Sigma\right)}^2\norm{u(t)}_{H^2\left(\Sigma\right)}^2\\
    &\quad+C\left(T,\norm{u_0}_{H^1\left(\Sigma\right)}\right)\left(1+\norm{u}_{H^2\left(\Sigma\right)}\right)\\
    \leq&-\dfrac14\int_{\Sigma}\abs{\nabla_{g}w(t)}_{g}^2\dif\mu_{g}+C\left(T,\norm{u_0}_{H^1\left(\Sigma\right)}\right)\left(1+\norm{w(t)}^2_{L^2\left(\Sigma\right)}\right)\left(1+\norm{\Delta_{g}u(t)}^2_{L^2\left(\Sigma\right)}\right).
\end{align*}
Thus
\begin{align*}
    \dfrac{\dif}{\dif t}\ln\left(1+\norm{\Delta_{g}u(t)}^2_{L^2\left(\Sigma\right)}+\int_0^t\int_{\Sigma}\abs{\nabla_{g}w(\tau)}_{g}^2\dif\mu_{g}\dif\tau\right)\leq C\left(T,\norm{u_0}_{H^1\left(\Sigma\right)}\right)\left(1+\norm{w(t)}^2_{L^2\left(\Sigma\right)}\right).
\end{align*}
Together with $u_0\in H^2(\Sigma)$, we obtain
\begin{align*}
    \ln\left(1+\norm{\Delta_{g}u(t)}^2_{L^2\left(\Sigma\right)}+\int_0^T\int_{\Sigma}\abs{\nabla_{g}w(\tau)}_{g}^2\dif\mu_{g}\dif\tau\right)\leq C\left(T,\norm{u_0}_{H^2\left(\Sigma\right)}\right)+C\left(T,\norm{u_0}_{H^1\left(\Sigma\right)}\right)\int_0^T\left(1+\norm{w(t)}^2_{L^2\left(\Sigma\right)}\right)\dif t.
\end{align*}
By \eqref{eq:functional}, we know that
\begin{align*}
    \int_0^T\norm{w(t)}^2_{L^2\left(\Sigma\right)}\dif t=\int_{0}^T\int_{\Sigma}e^{u(t)}\abs{\dfrac{\partial u(t)}{\partial t}}^2\dif\mu_{g}\dif t=J\left(u(0)-J\left(u(T)\right)\right)\leq C.
\end{align*}
Consequently, by using Sobolev embedding, we conclude
\begin{align*}%\label{eq:H_2}
    \norm{u(t)}_{H^2\left(\Sigma\right)}+\left(\int_0^T\norm{\dfrac{\partial u(t)}{\partial t}}_{H^1\left(\Sigma\right)}^2\dif t\right)^{1/2}\leq C\left(T,\norm{u_0}_{H^2\left(\Sigma\right)}\right).
\end{align*}

\item[Step 3]
\begin{quotation}
    $\left(\int_{0}^T\norm{\frac{\partial^2u}{\partial t^2}}_{H^{-1}\left(\Sigma\right)}^2\dif t\right)^{1/2}\leq C\left(T,\norm{u_0}_{H^2\left(\Sigma\right)}\right)$ for any $t\in[0,T]$.
\end{quotation}

Differential the equation \eqref{eq:mean-field-flow} with respect to $t$, we get
\begin{align*}
    e^u\ddot{u}+e^u\dot{u}^2=\Delta\dot{u}+8\pi\left(\frac{he^u\dot{u}}{\int_{\Sigma}he^u\dif\mu_{g}}-\frac{he^u\int_{\Sigma}he^u\dot{u}\dif\mu_{g}}{\left(\int_{\Sigma}he^u\dif\mu_{g}\right)^2}\right)
\end{align*}
where $\ddot u=\frac{\partial^2u}{\partial t^2}$ and $\dot u=\frac{\partial u}{\partial t}$.
Then for all $\psi\in H^1(\Sigma)$ with $\|\psi\|_{H^1(\Sigma)}\leq1$, we have
\begin{align*}
    \int_\Sigma\ddot{u}\psi\dif\mu_{g}
    =&\int_\Sigma e^{-u}\left(\Delta\dot{u}+8\pi\left(\frac{he^u\dot{u}}{\int_{\Sigma}he^u\dif\mu_{g}}-\frac{he^u\int_{\Sigma}he^u\dot{u}\dif\mu_{g}}{\left(\int_{\Sigma}he^u\dif\mu_{g}\right)^2}\right)\right)\psi\dif\mu_{g}-\int_\Sigma|\dot{u}|^2\psi\mu_{g}\\
    =&-\int_{\Sigma}\hin{\nabla_{g}\dot u}{\nabla_{g}\dot u\psi+\nabla_{g}\psi}_{g}e^{-u}\dif\mu_{g}+\int_\Sigma 8\pi\left(\dfrac{h\dot{u}}{\int_{\Sigma}he^u\dif\mu_{g}}-\dfrac{h\int_{\Sigma}he^u\dot{u}\dif\mu_{g}}{\left(\int_{\Sigma}he^u\dif\mu_{g}\right)^2}\right)\psi\dif\mu_{g}-\int_\Sigma|\dot{u}|^2\psi\mu_{g}\\
    \leq&C\left(T,\norm{u_0}_{H^2\left(\Sigma\right)}\right) \|\dot{u}\|_{H^1(\Sigma)}.
\end{align*}
Thus $\norm{\ddot u(t)}_{H^{-1}\left(\Sigma\right)}\leq C\left(T,\norm{u_0}_{H^2\left(\Sigma\right)}\right) \norm{\dot{u(t)}}_{H^1(\Sigma)}$ which implies the desired estimate.
\end{description}

Since we have the following embedding (cf. \cite[page 304, Theorem 2]{Eva10partial} and \cite[page 305, Theorem 3]{Eva10partial})
\begin{align*}
    C\left(0,T;H^1\left(\Sigma\right)\right)\subset H^1\left(0,T;H^1\left(\Sigma\right)\right),\quad C\left(0,T;L^2\left(\Sigma\right)\right)\subset L^2\left(0,T;H^1\left(\Sigma\right)\right)\cap H^1\left(0,T;H^{-1}\left(\Sigma\right)\right),
\end{align*}
we get 
\begin{align*}
u\in& \cap_{0<T<\infty}\left(L^{\infty}\left(0,T;H^2\left(\Sigma\right)\right)\cap H^1\left(0,T;H^1\left(\Sigma\right)\right)\cap H^2\left(0,T;H^{-1}\left(\Sigma\right)\right)\cap C\left(0,T;H^{1}\left(\Sigma\right)\right)\cap C^1\left(0,T;L^{2}\left(\Sigma\right)\right)\right).
\end{align*}
By using the parabolic Sobolev embedding theorems (cf. \cite[pages 368-369]{Cha89heat}) together with the interpolation inequality \eqref{eq:interpolation}, we get
\begin{align*}
    u\in\cap_{0< T<\infty} W^{2,1}_4\left(\Sigma\times[0,T]\right)\subset \cap_{0< T<\infty}C^{\alpha,\alpha/2}\left(\Sigma\times[0,T]\right),\quad\forall 0<\alpha<1.
\end{align*}
Here $W^{2,1}_p\left(\Sigma\times[0,T]\right)=L^p\left(0;T;W^{2,p}\left(\Sigma\right)\right)\cap W^{1,p}\left(0,T; L^p\left(\Sigma\right)\right)$ stands for the usual parabolic Sobolev space.

Then the standard regularity theory for parabolic equation gives \begin{align*}
    \norm{u(t)}_{C^{2+k+\alpha,(2+k+\alpha)/2}\left(\Sigma\times[0,T]\right)}\leq C\left(T,k,\norm{u_0}_{C^{2+k+\alpha}(\Sigma)}\right)
    \end{align*}
for all integer number $k\geq0$. In particular, we can extend this flow to infinity  and $u$ is smooth in $\Sigma\times(0,\infty)$.

Now assume $u_0\in H^2\left(\Sigma\right)$ and choose a sequence of smooth functions $u_{0,\varepsilon}$ on $\Sigma$ such that $u_{0,\varepsilon}$ converges to $u_0$ in $H^2\left(\Sigma\right)$ as $\varepsilon\to0$. Let $u_{\varepsilon}$ be the unique smooth solution to
\begin{align*}
    \begin{cases}
\dfrac{\partial e^{u_{\varepsilon}}}{\partial t}=\Delta_{g}u_{\varepsilon}+8\pi\left(\dfrac{he^{u_{\varepsilon}}}{\int_{\Sigma}he^{u_{\varepsilon}}\dif\mu_{g}}-1\right),&\Sigma\times(0,\infty),\\
u_{\varepsilon}(\cdot,0)=u_{0,\varepsilon},&\Sigma.
\end{cases}
\end{align*}
The a prior estimates \eqref{eq:a-priori}   gives the following estimates
\begin{align*}
    \esssup_{0\leq t\leq T}\norm{u_{\varepsilon}(t)}_{H^2\left(\Sigma\right)}+\left(\int_0^T\left(\norm{\dfrac{\partial u_{\varepsilon}(t)}{\partial t}}_{H^1\left(\Sigma\right)}^2+\norm{\dfrac{\partial^2 u_{\varepsilon}(t)}{\partial t^2}}_{H^{-1}\left(\Sigma\right)}^2\right)\dif t\right)^{1/2}\leq C\left(T,\norm{u_0}_{H^2\left(\Sigma\right)}\right),\quad\forall 0<T<\infty.
\end{align*}
Thus we obtain a solution $u\in C^{\infty}\left(\Sigma\times(0,\infty)\right)$ to \eqref{eq:mean-field-flow} with
\begin{align*}
u\in \cap_{0<T<\infty}\left(L^{\infty}\left(0,T;H^2\left(\Sigma\right)\right)\cap H^1\left(0,T;H^1\left(\Sigma\right)\right)\cap H^2\left(0,T;H^{-1}\left(\Sigma\right)\right)\right)
    \end{align*}
    and the desired a priori estimates \eqref{eq:a-priori}.

To prove the uniqueness of the solution, we assume that $u$ and $v$ are two solutions to \eqref{eq:mean-field-flow} with initial data $u_0$ and $v_0$ respectively. Denote $w=u-v$. By direct computations, we have
\begin{align}\label{eq:difference}
    a\dfrac{\partial w}{\partial t}=\Delta_{g}w+f-bw
\end{align}
where
\begin{align*}
    a=\int_0^1e^{su+(1-s)v}\dif s,\quad b=\int_0^1e^{su+(1-s)v}\left(s\dfrac{\partial u}{\partial t}+(1-s)\dfrac{\partial v}{\partial t}\right)\dif s=\dfrac{\partial a}{\partial t},\\
    f=8\pi\int_0^1\dfrac{he^{su+(1-s)v}}{\int_{\Sigma}he^{su+(1-s)v}\dif\mu_{g}}w\dif s
       -8\pi\int_0^1\dfrac{he^{su+(1-s)v}\int_{\Sigma}he^{su+(1-s)v}w\dif\mu_{g}}{\left(\int_{\Sigma}he^{su+(1-s)v\dif\mu_{g}}\right)^2}\dif s.
\end{align*}
One can check that there is a constant $C$ depends only on $T, \norm{u_0}_{H^2\left(\Sigma\right)}$ and $\norm{u_0}_{H^2\left(\Sigma\right)}$ such that for all $0\leq t\leq T$
\begin{align*}
    C^{-1}\leq a(t)\leq C, \quad\abs{b(t)}\leq C\left(\abs{\dfrac{\partial u(t)}{\partial t}}+\abs{\dfrac{\partial u(t)}{\partial t}}\right),\\
    \abs{f(t)}\leq C\left(\abs{w(t)}+\norm{w(t)}_{L^2\left(\Sigma\right)}\right),\quad f(t)w(t)\leq C w(t)^2.
\end{align*}
 Then we obtain
 \begin{align*}
     \dfrac{\dif}{\dif t}\int_{\Sigma}a(t)w(t)^2\dif\mu_{g}=&\int_{\Sigma}b(t)w(t)^2\dif\mu_{g}+2\int_{\Sigma}a(t)w(t)\dfrac{\partial w(t)}{\partial t}\dif\mu_{g}\\
     =&-2\int_{\Sigma}\abs{\nabla_{g}w(t)}_{g}^2\dif\mu_{g}+2\int_{\Sigma}f(t)w(t)\dif\mu_{g}-\int_{\Sigma}b(t)w(t)^2\dif\mu_{g}\\
    \leq&-\int_{\Sigma}\abs{\nabla_{g}w(t)}_{g}^2\dif\mu_{g}+C\int_{\Sigma}a(t)w(t)^2\dif\mu_{g}.
 \end{align*}
 Gronwall's inequality implies
 \begin{align}\label{eq:L2-w}
     \int_{\Sigma}w(t)^2\dif\mu_{g}\leq C\left[\int_{\Sigma}a(t)w(t)^2\dif\mu_{g}+\int_0^T\int_{\Sigma}\abs{\nabla_{g}w}_{g}^2\dif\mu_{g}\dif t\right]\leq C\norm{u_0-v_0}_{L^2\left(\Sigma\right)}^2,\quad\forall 0<t<T.
 \end{align}
 The uniqueness then follows from the above inequality and we finish the proof.

\end{proof}

\begin{rem}
One check that the difference of two solutions $u$ and $v$ satisfies
\begin{align*}
    \norm{u-v}_{W^{2,1}_2\left(\Sigma\times[0,T]\right)}\leq C\left(T,\norm{u_0}_{H^2\left(\Sigma\right)},\norm{v_0}_{H^2\left(\Sigma\right)}\right)\norm{u_0-v_0}_{H^2\left(\Sigma\right)}.
\end{align*}
The proof is standard. Roughly speaking, \eqref{eq:difference} implies
\begin{align*}
    \abs{a^{1/2}\dfrac{\partial w}{\partial t}-a^{-1/2}\Delta_{g}}^2=a\abs{\dfrac{\partial w}{\partial t}}^2+a^{-1}\abs{\Delta_{g}w}^2-2\hin{\dfrac{\partial w}{\partial t}}{\Delta_{g}w}_{g}.
\end{align*}
Integration by parts,
\begin{align*}
  \int_{\Sigma} \left(a\abs{\dfrac{\partial w}{\partial t}w}^2+a^{-1}\abs{\Delta_{g}w}^2\right)\dif\mu_{g}+\dfrac{\dif}{\dif t}\int_{\Sigma}\abs{\nabla_{g}w}_{g}^2\dif\mu_{g}\leq C\int_{\Sigma}\abs{w}^2\dif\mu_{g}+C\left(\int_{\Sigma}\abs{b}^4\dif\mu_{g}\right)^{1/2}\left(\int_{\Sigma}\abs{w}^4\dif\mu_{g}\right)^{1/2}.
\end{align*}
Applying the interpolation inequality \eqref{eq:interpolation} and the $L^2$-estimate \eqref{eq:L2-w} of the $w$, we have
\begin{align*}
   \int_0^T\int_{\Sigma} \left(a\abs{\dfrac{\partial w}{\partial t}w}^2+a^{-1}\abs{\Delta_{g}w}^2\right)\dif\mu_{g}\dif t+\max_{0\leq t\leq T}\int_{\Sigma}\abs{\nabla_{g}w(t)}_{g}^2\dif\mu_{g}\leq C\int_{\Sigma}\abs{u_0-v_0}^2\dif\mu_{g}
\end{align*}
where the constant $C$ depends only on $T,\norm{u_0}_{H^2\left(\Sigma\right)}$ and $\norm{v_0}_{H^2\left(\Sigma\right)}$. 

\end{rem}

\section{Blowup analysis}\label{sec:blowup}
In this section, we prove an estimate of a Dirac measure at the blowup points. Consequently, we show the fact that the flow develops at most one blowup point when the time goes to infinity.

According to \eqref{eq:functional} and \eqref{eq:functional-2}, we know that
\begin{align*}
    \int_{0}^{\infty}\int_{\Sigma}e^{u(t)}\abs{\dfrac{\partial u(t)}{\partial t}}^2\dif\mu_{g}\dif t\leq C.
\end{align*}
There is a sequence of positive numbers $\set{t_n}$ such that $n\leq t_n\leq n+1$ and
\begin{align*}
    \lim_{n\to\infty}\int_{\Sigma}e^{u(t_n)}\abs{\dfrac{\partial u(t_n)}{\partial t}}^2\dif\mu_{g}=0.
\end{align*}
Set 
\begin{align}\label{eq:model}
    u_n=u(t_n),\quad V_n=\dfrac{8\pi h}{\int_{\Sigma}he^{u_n}\dif\mu_{g}},\quad\rho=8\pi,\quad f_n=e^{u_n/2}\dfrac{\partial u(t_n)}{\partial t},
\end{align}
then
\begin{align}\label{eq:sequence}
    -\Delta_{g}u_n=V_ne^{u_n}-\rho-f_ne^{u_n/2},\quad\text{in}\ \Sigma,
\end{align}
and $u_n,V_n,\rho,f_n$ are smooth functions on $\Sigma$ satisfying
\begin{align}\label{eq:assumption}
    \rho>0,\quad 0\leq V_n\leq C,\quad\quad\lim_{n\to\infty}\norm{f_n}_{L^2\left(\Sigma\right)}=0.
\end{align}
One can check that
\begin{align}\label{eq:mass}
    \int_{\Sigma}e^{u_n}\dif\mu_{g}\leq C.
\end{align}

We say that a sequence $\set{u_n}$ which satisfies \eqref{eq:sequence} and \eqref{eq:assumption}  is a blowup sequence if $\limsup\limits_{n\to\infty}\max\limits_{\Sigma}u_n=+\infty$. %{\color{red}or equivalently $\limsup\limits_{n\to\infty}\|\nabla u_n\|_{L^2(\Sigma)}=+\infty$. }

\begin{lem}\label{compact sequence}
If $\set{u_n}$ is not a blowup sequence, then $\set{u_n}$ is bounded in $H^{2}\left(\Sigma\right)$.
\end{lem}
\begin{proof}
By definition, $\set{u_n^{+}}$ is bounded in $L^{\infty}\left(\Sigma\right)$. By the standard elliptic estimates and the normalization $\int_{\Sigma}\dif\mu_{g}=1$, we conclude that $\set{u_n-\bar u_n}$ is bounded in $H^{2}\left(\Sigma\right)$, where $\bar{u}_n\coloneqq\bar{u}(t_n)=\int_{\Sigma}u(t_n)\dif\mu_{g}$. By Jensen's inequality, according to \eqref{eq:mass}, we have $\bar u_n\leq C$. It suffices to prove that $\bar u_n\geq-C$. Otherwise, there is a subsequence $\set{u_{n_k}}$ such that $\lim_{n_k\to\infty}\bar u_{n_k}=-\infty$. Notice that
\begin{align*}
-\Delta_{g}\left(u_{n_k}-\bar u_{n_k}\right)=V_{n_k} e^{\bar u_{n_k}}e^{u_{n_k}-\bar u_{n_k}}-\rho-f_{n_k}e^{\bar u_{n_k}/2}e^{\left(u_{n_k}-\bar u_{n_k}\right)/2},\quad\text{in}\ \Sigma.
\end{align*}
We may assume $u_{n_k}-\bar u_{n_k}$ converges weakly to $\hat u$ in $H^2\left(\Sigma\right)$ and strongly in $L^1\left(\Sigma\right)$. Then $\set{e^{p\left(u_{n_k}-\bar u_{n_k}\right)}}$ converges strongly to $e^{p\hat u}$ in $L^1\left(\Sigma\right)$ for each $p>0$. Thus $\hat u$ is a weak solution to
\begin{align*}
-\Delta_{g}\hat u=-\rho.
\end{align*}
It is well know that $\hat u\in C^{\infty}\left(\Sigma\right)$ and $\rho=0$ which is a contradiction. Therefore, $\set{u_n}$ is bounded in $H^2\left(\Sigma\right)$.
\end{proof}

From now on, we assume $\set{u_n}$ is a blowup sequence. Since $\set{V_n e^{u_n}}$ is bounded in $L^1\left(\Sigma\right)$, we may assume $\set{V_n e^{u_{n}}\dif\mu_{g}}$  converges to a nonzero Radon measure $\mu$ on $\Sigma$  in the sense of  measures. By using the method of potential estimates (cf. \cite[Lemma 7.12]{GilTru01}), we get
\begin{align*}
\norm{u_n-\bar u_n}_{W^{1,p}\left(\Sigma\right)}\leq C_p,\quad\forall 1\leq p<2.
\end{align*}
We may assume  $\set{u_n-\bar u_n}$ converges weakly to $G$  in $W^{1,p}\left(\Sigma\right)$ and strongly in $L^p\left(\Sigma\right)$ for every $1<p<2$. Hence $U$ satisfies
\begin{align*}
\begin{cases}
-\Delta_{g}G=\mu-\rho,\quad\text{in}\ \Sigma,\\
\int_{\Sigma}G\dif\mu_{g}=0,
\end{cases}
\end{align*}
in the sense of distribution. 
Define the singular set $S$ of the sequence $\set{u_n}$ as follows
\begin{align*}
S=\set{x\in\Sigma:\mu\left(\set{x}\right)\geq 4\pi.}
\end{align*}
It is easy to check that $S$ is a finite nonempty subset of $\Sigma$.

Recall Brezis-Merle's estimate (\cite[Theorem 1]{BreMer91}).

\begin{lem}[cf. \cite{DinJosLiWan97}]\label{lem:BM}Let $\Omega\subset\Sigma$ be a smooth domain. Assume $u$ is a solution to
\begin{align*}
\begin{cases}
    -\Delta_{g}u=f,&\text{in}\ \Omega,\\
    u=0,&\text{on}\ \partial\Omega,
\end{cases}
\end{align*}
where $f\in L^1\left(\Omega\right)$. For every $0<\delta<4\pi$, there is a constant $C$ depending only on $\delta$ and $\Omega$ such that
\begin{align*}
    \int_{\Omega}\exp\left(\dfrac{\left(4\pi-\delta\right)\abs{u}}{\norm{f}_{L^1\left(\Omega\right)}}\right)\dif\mu_{g}\leq C.
\end{align*}
\end{lem}

As a consequence, we have the following Lemma (cf. \cite[Lemma 2.8]{DinJosLiWan97}).

\begin{lem}\label{lem:basic}
If $x\notin S$, then there is a geodesic ball $B^{g}_r(x)\subset\Sigma\setminus S$ and a positive constant $C=C_x$ such that
\begin{align*}
\norm{u_n-\bar u_n}_{L^{\infty}\left(B^{g}_r(x)\right)}\leq C.
\end{align*}
\end{lem}
\begin{proof}There exist $\delta=\delta_x\in\left(0,2\pi\right), r=r_x\in\left(0,\mathrm{inj}\left(\Sigma\right)/4\right), N=N_x\in\mathbb{N}$ such that
\begin{align*}
\int_{B^{g}_{4r}(x)}\abs{V_n e^{u_n}-f_ne^{u_n/2}}\dif\mu_{g}\leq 4\pi-2\delta,\quad\forall n\geq N.
\end{align*}
Solve
\begin{align*}
\begin{cases}
-\Delta_{g}y_n=-\rho,&\text{in}\ B^{g}_{4r}(x),\\
y_n=0,&\text{on}\ \partial B^{g}_{4r}(x).
\end{cases}
\end{align*}
It is well know that $\set{y_n}$ is bounded in $L^{\infty}\left(B^{g}_{4r}(x)\right)$. 
Solve
\begin{align}\label{eq:w_n}
\begin{cases}
-\Delta_{g}w_n=V_n e^{u_n}-f_ne^{u_n/2},&\text{in}\ B^{g}_{4r}(x),\\
w_n=0,&\text{on}\ \partial B^{g}_{4r}(x).
\end{cases}
\end{align}
According to \autoref{lem:BM}, we have
\begin{align*}
\norm{e^{\abs{w_n}}}_{L^{p}\left(B^{g}_{4r}(x)\right)}\leq C,\quad p=\dfrac{4\pi-\delta}{4\pi-2\delta}>1.
\end{align*}
In particular, $\set{w_n}$ is bounded in $L^1\left(B^{g}_{4r}(x)\right)$. Since $h_n\coloneqq u_n-\bar u_n-y_n-w_n$ is harmonic in $B^{g}_{4r}(x)$, we have
\begin{align*}
\norm{h_n}_{L^{\infty}\left(B^{g}_{2r}(x)\right)}\leq &C\norm{h_n}_{L^{1}\left(B^{g}_{4r}(x)\right)}\\
\leq&C\left(\norm{u_n-\bar u_n}_{L^{1}\left(B^{g}_{4r}(x)\right)}+\norm{y_n}_{L^{1}\left(B^{g}_{4r}(x)\right)}+\norm{w_n}_{L^{1}\left(B^{g}_{4r}(x)\right)}\right)\\
\leq&C\left(\norm{u_n-\bar u_n}_{L^{1}\left(\Sigma\right)}+\norm{w_n}_{L^{1}\left(B^{g}_{4r}(x)\right)}+\norm{y_n}_{L^{1}\left(B^{g}_{4r}(x)\right)}\right)\\
\leq&C.
\end{align*}
Thus
\begin{align*}
\norm{e^{u_n}}_{L^p\left(B_{2r}^{g}(x)\right)}\leq C.
\end{align*}
Applying the standard elliptic estimates for \eqref{eq:w_n}, we get
\begin{align*}
\norm{w_n}_{L^{\infty}\left(B^g_{r}(x)\right)}\leq C.
\end{align*}
Hence
\begin{align*}
\norm{u_n-\bar u_n}_{L^{\infty}\left(B^g_{r}(x)\right)}\leq C.
\end{align*}
\end{proof}

\begin{theorem}If $\set{u_n}$ is a blowup sequence, then $S$ is nonempty and
\begin{align*}
    S=\set{x\in\Sigma: \exists\  \set{x_n}\subset\Sigma,\  \lim_{n\to\infty}x_n=x,\  \lim_{n\to\infty}u_n\left(x_n\right)=+\infty.}
\end{align*}
Moreover $\lim_{n\to\infty}\bar u_n=-\infty$. Thus $\set{u_n}$ converges to $-\infty$ uniformly on compact subsets of $\Sigma\setminus S$ and $\mu=\sum_{x\in S}\mu\left(\set{x}\right)\delta_{x}$ is a Dirac measure.
\end{theorem}
\begin{proof}
According to \autoref{lem:basic}, we know that $\set{u_n-\bar u_n}$ is bounded in $L^{\infty}_{loc}\left(\Sigma\setminus S\right)$. 

If $S=\emptyset$, then $\set{u_n-\bar u_n}$ is bounded in $L^{\infty}\left(\Sigma\right)$ which implies that $\set{u_n^{+}}$ is bounded in $L^{\infty}\left(\Sigma\right)$ which is a contradiction.

We claim that $\lim_{n\to\infty}\bar u_n=-\infty$. Otherwise, there is a subsequence of $\set{u_n}$ which also denoted by $\set{u_n}$ such that
\begin{align*}
   \bar u_n\geq-C.
\end{align*}
For $x\in S$, choose $r>0$ such that $B^{g}_{2r}(x)\cap S=\set{x}$.
According to \autoref{lem:basic}, $\set{u_n}$ is bounded in $L^{\infty}_{loc}\left(B^{g}_{2r}(x)\setminus\set{x}\right)$. In particular, $M\coloneqq\sup_{n}\norm{u_n}_{L^{\infty}\left(\partial B^{g}_{r}\left(x\right)\right)}<\infty$. Solve
\begin{align*}
\begin{cases}
    -\Delta_{g} z_n=V_ne^{u_n}-\rho-f_ne^{u_n/2},&\text{in}\ B^{g}_{r}(x),\\
    z_n=-M,&\text{on}\ B^{g}_{r}(x).
\end{cases}
\end{align*}
By potential estimates, we know that $z_n$ is bounded in $W^{1,p}\left(B^{g}_{r}(x)\right)$ for all $1<p<2$. Thus, up to a subsequence, $z_n$ converges weakly to $z\in W^{1,p}\left(B^{g}_{r}(x)\right)$ for all $1<p<2$ and strongly in $L^q\left(B^{g}_{r}(x)\right)$ for all $1<q<\infty$. Then $z$ is a weak solution to
\begin{align*}
\begin{cases}
    -\Delta_{g} z=\mu\left(\set{x}\right)\delta_{x}-\rho,&\text{in}\ B^{g}_{r}(x),\\
    z=-M,&\text{on}\ B^{g}_{r}(x).
\end{cases}
\end{align*}
Thus
\begin{align*}
    z(\cdot)\geq-\dfrac{\mu\left(\set{x}\right)}{2\pi}\ln\mathrm{dist}_{g}(\cdot,x)-C.
\end{align*}
Since $\mu\left(\set{x}\right)\geq4\pi$, we get
\begin{align*}
    \int_{B^{g}_{r}(x)}e^{z}\dif\mu_{g}=\infty.
\end{align*}
On the other hand, the maximum principle implies that $z_n\leq u_n$. By Fatou's Lemma,
\begin{align*}
    \infty=\int_{B^{g}_{r}(x)}e^{z}\dif\mu_{g}\leq\liminf_{n\to\infty}\int_{B^{g}_{r}(x)}e^{z_n}\dif\mu_{g}\leq\liminf_{n\to\infty}\int_{B^{g}_{r}(x)}e^{u_n}\dif\mu_{g}\leq C,
\end{align*}
which is a contradiction.

Hence $\set{u_n}$ converges to $-\infty$ uniformly on compact subsets of $\Sigma\setminus S$. Thus for every domain $\Omega\subset\Sigma$
\begin{align*}
    \mu\left(\Omega\right)=&\lim_{n\to\infty}\int_{\Omega}V_n e^{u_n}\dif\mu_{
    g}\\
    =&\sum_{x\in S}\lim_{r\to0}\lim_{n\to\infty}\int_{\Omega\setminus B_{r}^{g}(x)}V_n e^{u_n}\dif\mu_{
    g}+\sum_{x\in S}\lim_{r\to0}\lim_{n\to\infty}\int_{\Omega\cap B_{r}^{g}(x)}V_n e^{u_n}\dif\mu_{
    g}\\
    =&\sum_{x\in S\cap\Omega}\lim_{r\to0}\lim_{n\to\infty}\int_{ B_{r}^{g}(x)}V_n e^{u_n}\dif\mu_{
    g}\\
    =&\sum_{x\in S\cap\Omega}\mu\left(\set{x}\right)\\
    =&\mu\left(\Omega\cap S\right).
\end{align*}
In other words, $\mu=\sum_{x\in S}\mu\left(\set{x}\right)\delta_{x}$ is a Dirac measure.

According to \autoref{lem:basic}, we obtain
\begin{align*}
    \set{x\in\Sigma: \exists\  \set{x_n}\subset\Sigma,\  \lim_{n\to\infty}x_n=x,\  \lim_{n\to\infty}u_n\left(x_n\right)=+\infty.}\subset S.
\end{align*}
For $x_0\in S$, choose a geodesic ball  $B_{2r}^{g}\left(x_0\right)$ such that $B_{2r}^{g}\left(x_0\right)\cap S=\set{x_0}$. Choose $x_n\in\overline{B_{r}^{g}\left(x_0\right)}$ such that
\begin{align*}
    \lambda_n\coloneqq\max_{\overline{B_{r}^{\Sigma}\left(x_0\right)}}u_n=u_n\left(x_n\right).
\end{align*}
\begin{description}
     \item [Case 1.] $\lim_{n\to\infty}\lambda_n=\infty$. 
    
    Otherwise, up to a subsequence,  $\set{u_n^{+}}$ is bounded in $L^{\infty}\left(B^{g}_r\left(x_0\right)\right)$. Thus $\set{e^{u_n}}$ is bounded in $L^{\infty}\left(B^{g}_r\left(x_0\right)\right)$ which is a contradiction.
    
    \item [Case 2.]  $\lim_{n\to\infty}x_n=x_0$.
    
    Otherwise, up to a subsequence, $\lim_{n\to\infty}x_n=\tilde x\in\overline{B_{r}\left(x_0\right)}\setminus\set{x_0}$. Thus $\tilde x$ is not a singular point which is impossible according \autoref{lem:basic} and the above claim.
\end{description}
Consequently,
\begin{align*}
    S\subset\set{x\in\Sigma: \exists\  \set{x_n}\subset\Sigma,\  \lim_{n\to\infty}x_n=x,\  \lim_{n\to\infty}u_n\left(x_n\right)=+\infty.}
\end{align*}

\end{proof}

Now we want to prove that $\mu\left(\set{x_0}\right)\geq8\pi$. We assume additionally that $V_n$ converges to $V$ in $C^0\left(\Sigma\right)$.
\begin{lem}For each $x_0\in S$, we have $V(x_0)>0$ and  $\mu\left(\set{x_0}\right)\geq8\pi$.
\end{lem}
\begin{proof}
Assume $B^{g}_{2r}(x_0)\cap S=\set{x_0}$. Choose $x_n\in B^{g}_{2r}(x_0)$ such that
\begin{align*}
\lambda_n\coloneqq\max_{\overline{B^{g}_{r}(x_0)}}u_n=u_{n}\left(x_n\right).
\end{align*}
It is easy to check that
\begin{align*}
\lim_{n\to\infty}\lambda_n=+\infty,\quad\lim_{n\to\infty}x_n=x_0.
\end{align*}
Now choose a conformal coordinate $\set{x}$ centered at $x_0$. We have $g=e^{\phi(x)}\abs{\dif x}^2$ and
\begin{align*}
-\Delta_{\mathbb{R}^2}u_n=V_n e^{\phi} e^{u_n}-e^{\phi}\rho-f_ne^{\phi}e^{u_n/2},\quad\abs{x}<2\tilde r.
\end{align*}
Consider 
\begin{align*}
\tilde u_n(x)=u_n\left(x_n+e^{-\lambda_n/2}x\right)-\lambda_n,
\end{align*}
then for $\abs{x}<e^{\lambda_n/2}\tilde r$,
\begin{align*}
-\Delta_{\mathbb{R}^2}\tilde u_n(x)=V_n\left(x_n+e^{-\lambda_n/2}x\right) e^{\phi\left(x_n+e^{-\lambda_n/2}x\right)}e^{\tilde u_n(x)}-e^{\phi\left(x_n+e^{-\lambda_n/2}x\right)}\rho-f_n\left(x_n+e^{-\lambda_n/2}x\right)e^{\phi\left(x_n+e^{-\lambda_n/2}x\right)-\lambda_n/2}e^{\tilde u_n(x)/2}.
\end{align*}
We have $\tilde u_n\leq0, \tilde u_n(0)=0$ and
\begin{align*}
\int_{B_{e^{\lambda_n/2}\tilde r}}e^{\tilde u_n}\dif\mu_{\mathbb{R}^2}\leq \int_{B^{g}_{2\tilde r}(x_0)}e^{u_n}\dif\mu_{g}\leq C,\\
\int_{B_{e^{\lambda_n/2}\tilde r}}\abs{f_n\left(x_n+e^{-\lambda_n/2}x\right)e^{\phi\left(x_n+e^{-\lambda_n/2}x\right)-\lambda_n/2}}^2\dif\mu_{\mathbb{R}^2}\leq\int_{B^{g}_{2\tilde r}(x_0)}f_n^2\dif\mu_{g}\to0.
\end{align*}
Thus up to a subsequence, $\set{\tilde u_n}$ converges weakly to $\tilde u_{\infty}$ in $H^{2}_{loc}\left(\mathbb{R}^2\right)$ and strongly in $H^{1}_{loc}\left(\mathbb{R}^2\right)$. In particular, $\tilde u_{\infty}$ is a weak solution to
\begin{align*}
\begin{cases}
-\Delta_{\mathbb{R}^2}\tilde u_{\infty}=V(x_0) e^{\phi(0)}e^{\tilde u_{\infty}},&\text{in}\ \mathbb{R}^2,\\
\int_{\mathbb{R}^2}e^{\tilde u_{\infty}}\dif\mu_{\mathbb{R}^2}<\infty.
\end{cases}
\end{align*}
By a classification theorem of Chen-Li \cite{CheLi91}, we know that 
\begin{align*}
\int_{\mathbb{R}^2}V(x_0) e^{\phi(0)}e^{\tilde u_{\infty}}\dif\mu_{\mathbb{R}^2}=8\pi.
\end{align*}
In particular $V(x_0)>0$. By Fatou's Lemma, we have
\begin{align*}
\int_{\mathbb{R}^2}V(x_0) e^{\phi(0)}e^{\tilde u_{\infty}}\dif\mu_{\mathbb{R}^2}=&\lim_{R\to\infty}\int_{B_{R}}V(x_0) e^{\phi(0)}e^{\tilde u_{\infty}}\dif\mu_{\mathbb{R}^2}\\
\leq&\lim_{R\to\infty}\liminf_{n\to\infty}\int_{B_{R}}V_n\left(x_n+e^{-\lambda_n/2}x\right)e^{\phi\left(x_n+e^{-\lambda_n/2}x\right)}e^{\tilde u_n(x)}\dif\mu_{\mathbb{R}^2}\\
=&\lim_{R\to\infty}\liminf_{n\to\infty}\int_{B^{g}_{e^{-\lambda_n/2}R}\left(x_n\right)}V_n\left(x_n+e^{-\lambda_n/2}x\right) e^{u_n}\dif\mu_{g}\\
\leq&\lim_{r\to0}\liminf_{n\to\infty}\int_{B^{g}_{r}\left(x_0\right)}V e^{u_n}\dif\mu_{g}\\
=&\mu\left(\set{x_0}\right).
\end{align*}
Thus
\begin{align*}
\mu\left(\set{x_0}\right)\geq 8\pi.
\end{align*}
\end{proof}

In our initial model \eqref{eq:model}, we must have $\mu\left(\set{x_0}\right)=8\pi$ and $\#S=1$. Moreover,
\begin{align*}
    V=\lim_{n}V_n=\lim_{n\to\infty}\dfrac{8\pi h}{\int_{\Sigma}he^{u_n}\dif\mu_{g}}=\dfrac{8\pi h}{h(x_0)\int_{\Sigma}e^{u_0}\dif\mu_{g}}
\end{align*}
and $h(x_0)>0$.

\section{Lower bound for the functional}\label{sec:lower}

In this section, we give a lower bound for  $J(u(t))$ along the flow, i.e. we  give the proof of \autoref{thm:lower-bounds}.
\begin{proof}[Proof of \autoref{thm:lower-bounds}]
Suppose our flow develops a singularity as time goes to infinity, we will analyse the asymptotic behavior of the flow near and away from the blow-up point and derive a lower bound of $J(u)$. From the previous compactness argument, there is only one blow-up point when $\rho=8\pi$, denoted by $x_0$. Then there is a sequence of points $\{x_n\}$ such that
\begin{align*}
    \lim_{n\to\infty}x_n=x_0, \ \lambda_n=u_n(x_n)=\max_{\Sigma}u_n=\max_{\Sigma}u(t_n)=+\infty,
\end{align*}where $t_n\to\infty$ as $n\to\infty$. 
 In an isothermal coordinate system $\{x\}$ around $x_0$, we still denote $u_n$ and $x_n$ in this coordinate by $u_n$ and $x_n$, respectively. Set $r_n=e^{-\lambda_n/2}$ and
\begin{align*}
    \tilde{u}_n\coloneqq u_n(x_n+r_n x)-\lambda_n. 
\end{align*}
Then we have $\tilde{u}_n$ weakly converges to $\tilde{u}_{\infty}$ satisfying
\begin{align*}
  \tilde{u}_{\infty}=-2\ln(1+a|x|^2), \ a=\frac{\pi e^{\phi(0)}}{\int_\Sigma e^{u_0}\dif\mu_{g}}
\end{align*}
and 
\begin{align}\label{inside}
\begin{split}
    \lim_{n\to\infty}\frac12\int_{B^g_{r_nR}(x_n)}|du_n|_g^2d\mu_g&=\lim_{n\to\infty}\frac12\int_{B_R(x_n)}|d\tilde{u}_n|^2d\mu_{\mathbb{R}^2}\\
    &=\pi\int_{0}^R\left|\frac{4ar}{1+ar^2}\right|^2rdr\\
    &=8\pi\int_{0}^{aR^2}\frac{s}{(1+s)^2}ds\\
    &=8\pi\int_{0}^{aR^2}\left(\frac{1}{1+s}-\frac{1}{(1+s)^2}\right)ds\\
    &=8\pi\left(\ln(1+aR^2)+\frac{1}{1+aR^2}-1\right)\\
    &=8\pi\ln(aR^2)-8\pi+o_R(1).
    \end{split}
\end{align}
Here and in the following, we use $o_{R}(1),o_{n}(1),  o_{\delta}(1)$ to denote those functions which converges to zero as $R\to+\infty, n\to\infty, \delta\to0$ respectively.

Since $u_n-\bar{u}_n$ converges to $G$ weakly in $W^{1,p}(\Sigma)$ for $1<p<2$ and strongly in $H^2_{loc}(\Sigma\setminus\{x_0\})$ (see Proposition 3.5 in \cite{li2019convergence}) and G satisfies
\begin{align*}
\begin{cases}
-\Delta_gG=8\pi(\delta_{x_0}-1),&\Sigma,\\
\int_\Sigma Gd\mu_g=0.
\end{cases}
\end{align*}
we get
\begin{align*}
\begin{split}
    \lim_{n\to\infty}\frac12\int_{\Sigma\setminus B^g_\delta(x_n)}|du_n|^2d\mu_g
    &=\lim_{n\to\infty}\frac12\int_{\Sigma\setminus B^g_\delta(x_0)}|du_n|^2d\mu_g\\
    &=\frac12\int_{\Sigma\setminus B^g_\delta(x_0)}|dG|^2d\mu_g\\
    &=-\frac12\int_{\Sigma\setminus B^g_\delta(x_0)}G\Delta_gGd\mu_g-\frac12\int_{\partial B^g_\delta(x_0)}G\frac{\partial}{\partial\nu}Gd\mu_g\\
    &=4\pi\int_{B^g_\delta(x_0)}G\dif\mu_{g}-\frac12\int_{\partial B^g_\delta(x_0)}G\frac{\partial}{\partial\nu}Gd\mu_g,
    \end{split}
\end{align*}
where $\nu$ is the normal vector field on $\partial B^g_\delta(x_0)$ pointing to the complement of $B^g_\delta(x_0)$.

In normal coordinate, $G$ has the following expansion
\begin{align*}
    G(x)=-4\ln|x-x_0|+A(x_0)+(b,x-x_0)+(x-x_0)^Tc(x-x_0)+O(|x-x_0|^3).
\end{align*}
Then 
\begin{align}\label{outside}
    \lim_{n\to\infty}\frac12\int_{\Sigma\setminus B^g_\delta(x_n)}|du_n|^2d\mu_g=-16\pi\ln\delta+4\pi A(x_0)+o_\delta(1).
\end{align}

Define
\begin{align*}
    u_n^*(r)=\frac{1}{2\pi}\int_0^{2\pi}u_n(x_n+re^{i\theta})d\theta.
\end{align*}
Then for $r_nR\leq r<s\leq\delta$, we have
\begin{align*}
    \int_{B_s\setminus B_r}|du_n^*|^2dx\leq\int_{B_s\setminus B_r}\left|\frac{\partial u_n}{\partial r}\right|^2dx.
\end{align*}
Notice that 
\begin{align}\label{average limit}
\begin{split}
    &\lim_{n\to\infty}(u_n^*(r_nR)+2\ln r_n)=-2\ln(1+aR^2)=-2\ln(aR^2)+o_R(1),\\
    &\lim_{n\to\infty}(u_n^*(\delta)-\bar u_n)=\frac{1}{2\pi}\int_0^{2\pi}G(\delta e^{i\theta})d\theta=-4\ln\delta+A(x_0)+o_\delta(1).
\end{split}
\end{align}
Let $w_n$ be the harmonic functions in the neck domains $B^g_\delta(x_n)\setminus B^g_{r_nR}(x_n)$ such that
\begin{align*}
    w_n|_{\partial B^g_\delta(x_n)}=u_n^*(\delta), \ w_n|_{\partial B^g_{r_nR}(x_n)}=u_n^*(r_nR).
\end{align*}
Then we have
\begin{align*}
\begin{split}
    \frac12\int_{B^g_\delta(x_n)\setminus B^g_{r_nR}(x_n)}|du_n|^2d\mu_g
    &\geq\frac12\int_{B^g_\delta(x_n)\setminus B^g_{r_nR}(x_n)}|du_n^*|^2d\mu_g\\
    &\geq\frac12\int_{B^g_\delta(x_n)\setminus B^g_{r_nR}(x_n)}|dw_n|^2d\mu_g\\
    &\geq\frac{\pi(u_n^*(\delta)-u_n^*(r_nR))^2}{\ln\delta-\ln(r_nR)}.    
\end{split}
\end{align*}

Set $\tau_n\coloneqq u_n^*(\delta)-u_n^*(r_nR)-\bar{u}_n-2\ln r_n$. It follows from \eqref{average limit} that
\begin{align*}
    \lim_{n\to\infty}\tau_n=-4\ln\delta+A(x_0)+2\ln(aR^2)+o_R(1)+o_\delta(1).
\end{align*}
Then we get
\begin{equation}\label{neck}
    \begin{split}
           \frac12\int_{B^g_\delta(x_n)\setminus B^g_{r_nR}(x_n)}|du_n|^2d\mu_g
           &\geq\frac{\pi(\tau_n+\bar{u}_n+2\ln r_n)^2}{-\ln r_n}\left(1-\frac{\ln(R/\delta)}{-\ln r_n}\right)^{-1}\\
           &\geq\frac{\pi(\tau_n+\bar{u}_n-2\ln r_n)^2}{-\ln r_n}-8\pi\bar{u}_n+32\pi\ln\delta-8\pi A(x_0)-16\ln(aR^2)\\
           &\quad+\pi\left(2+\frac{\tau_n}{\ln r_n}+\frac{\bar{u}_n}{\ln r_n}\right)^2\ln(R/\delta)+o_R(1)+o_\delta(1)
    \end{split}
\end{equation}
for large $n$.

Thus, \eqref{inside}, \eqref{outside} and \eqref{neck} give us 
\begin{align*}
\begin{split}
    J(u_n)&\geq 8\pi\ln(aR^2)-8\pi-16\pi\ln\delta+4\pi A(x_0)+8\pi\left(\bar{u}_n-\ln h(x_0)-\ln\int_\Sigma e^{u_0}\dif\mu_{g}\right)\\
    &\quad+\frac{\pi(\tau_n+\bar{u}_n-2\ln r_n)^2}{-\ln r_n}-8\pi\bar{u}_n+32\pi\ln\delta-8\pi A(x_0)-16\ln(aR^2)\\
           &\quad+\pi\left(2+\frac{\tau_n}{\ln r_n}+\frac{\bar{u}_n}{\ln r_n}\right)^2\ln(R/\delta)+o_R(1)+o_\delta(1)+o_n(1)\\
           &=-4\pi(A(x_0)+2\ln h(x_0))-8\pi\ln\pi-8\pi+o_R(1)+o_\delta(1)+o_n(1)\\
           &\quad+\frac{\pi(\tau_n+\bar{u}_n-2\ln r_n)^2}{-\ln r_n}+\pi\left[\left(2+\frac{\tau_n}{\ln r_n}+\frac{\bar{u}_n}{\ln r_n}\right)^2-16\right]\ln(R/\delta).
\end{split}
\end{align*}
Since $J(u_n)\leq J(u_0)$, we have $\lim_{n\to\infty}\bar{u}_n\to2\ln r_n$. Hence,
\begin{align*}
    \lim_{n\to\infty}J(u_n)\geq-4\pi\max_{x\in\Sigma}(A(x)+2\ln h(x))-8\pi\ln\pi-8\pi.
\end{align*}
By the monotonicity formula \eqref{eq:functional}, we conclude that
\begin{align*}
    J(u(t))\geq C_0=-4\pi\max_{x\in\Sigma}(A(x)+2\ln h(x))-8\pi\ln\pi-8\pi,\quad\forall t\geq0.
\end{align*}

\end{proof}

\section{Global convergence}\label{sec:convergence}

\begin{proof}[Proof of \autoref{thm:converges}]Notice that there is a sequence of positive numbers $\set{t_n}$ such that $n\leq t_n\leq n+1$ and
\begin{align*}
    \lim_{n\to\infty}\int_{\Sigma}e^{u(t_n)}\abs{\dfrac{\partial u(t_n)}{\partial t}}^2\dif\mu_{g}=0.
\end{align*}
By the lower bound of $J$ along the flow stated in \autoref{thm:lower-bounds}, the existence of mean field equation \eqref{eq:kw8pi} is reduced to construct a function whose value under $J$ is strictly less than $C_0$. In fact, such kind of functions were constructed in \cite{DinJosLiWan97} provided that 
\begin{align*}
\begin{split}
    &\Delta_{g} h(p_0)+2(b_1(p_0)k_1(p_0)+b_2(p_0)k_2(p_0))\\
    &>-\left(8\pi+b_1^2(p_0)+b_2^2(p_0)-2K(p_0)\right)h(p_0),
\end{split}
\end{align*}
where $K(x)$ is the Gaussian curvature of $\Sigma$, $\nabla h(p_0)=(k_1(p_0),k_2(p_0))$ in the normal coordinate system, $p_0$ is the maximum point of $A(q)+2\ln h(q)$ and $b_1(p_0)$, $b_2(p_0)$ are the constants in the following expression of Green function $G$:
\begin{align*}
    G(x,p_0)=-4\ln r+A(p_0)+b_1(p_0)x_1+b_2(p_0)x_2+c_1x_1^2+2c_2x_1x_2+c_3x_2^2+O(r^3),
\end{align*}
where $r(x)=\mathrm{dist}_g(x,p_0)$. The sequence $\set{u_n}$ can not blowup by our assumption. 
By \autoref{compact sequence}, $\set{u_n}$ is bounded in $H^2(\Sigma)$ and there is a function $u_\infty\in H^2(\Sigma)$ and a subsequence $\set{u_{n_k}}$ of $\set{u_{n}}$  such that 
\begin{align*}
    u_{n_k}\to u_\infty \quad \text{weakly in} \quad H^2(\Sigma)
\end{align*}
and
\begin{align*}
    u_{n_k}\to u_\infty \quad \text{in} \quad C^\alpha(\Sigma)
\end{align*}
for $\alpha\in(0,1)$ as $n_k\to\infty$.
It is easy to see that $u_\infty$ is a smooth solution to 
\begin{align*}
    -\Delta_{g} u_{\infty}+8\pi=8\pi\frac{he^{u_{\infty}}}{\int_\Sigma he^{u_{\infty}}\dif\mu_{g}}.
\end{align*}
 To obtain the strong convergence for $\set{u_{n_k}}$, please notice that 
\begin{align*}
\begin{split}
&\int_{\Sigma}\left(\Delta_{g}u_{n_k}-\Delta_{g}u_{\infty}\right)^2\\
&=\int_{\Sigma}\left(\frac{\partial{e^{u_{n_k}}}}{\partial{t}}+8\pi\left(\dfrac{he^{u_\infty}}{\int_{\Sigma}he^{u_\infty}\dif\mu_{g}}-\dfrac{he^{u_{n_k}}}{\int_{\Sigma}he^{u_{n_k}}\dif\mu_{g}}\right)\right)^2\dif\mu_{g}\\
&\leq C\int_{\Sigma}\left(e^{u_\infty}-e^{u_{n_k}}\right)^2\dif\mu_{g}+C\int_{\Sigma}\bigg|\frac{\partial{u_{n_k}}}{\partial{t}}\bigg|^2e^{u_{n_k}}\dif\mu_{g}\to 0
\end{split}
\end{align*}
as $n_k\to+\infty$.

We use {\L}ojasiewicz-Simon gradient inequality to get the global convergence of the flow. When $h>0$, one can refer to  \cite{casteras2015mean} for non-critical cases, i.e. $\rho\neq8k\pi$ and \cite{li2019convergence} for $\rho=8\pi$. In both papers, the authors just provided the paper by Simon \cite{simon1983asymptotics} and no more details were given. In this section, we give a detailed proof and some references. We divide the proof of the global convergence to several steps.

\begin{description}
\item[Step 1]
\begin{quotation}
    $\norm{u(t)^+}_{L^{\infty}\left(\Sigma\right)}\leq C.$
\end{quotation} 

Since
\begin{align*}
\dfrac{\partial u}{\partial t}\leq e^{-u}\Delta_{g}u+C.
\end{align*}
Applying the maximum principle, we have $\frac{\dif}{\dif t}\left(\max_{\Sigma}u(t)-Ct\right)\leq0$. By using the fact $\set{u_{n}}$ is bounded in $L^{\infty}\left(\Sigma\right)$ and $n\leq t_n\leq n+1$, we conclude that $u(t)^{+}$ is bounded in $L^{\infty}\left(\Sigma\right)$.

 \item[Step 2]
 \begin{quotation}
     $\norm{u(t)}_{H^1\left(\Sigma\right)}\leq C$.
 \end{quotation}
     
      Denote
\begin{align*}
A(t)=\set{x\in\Sigma: e^{u(t)}\geq\dfrac12\int_{\Sigma}e^{u_0}\dif\mu_{g}}.
\end{align*}
Then
\begin{align*}
\int_{A(t)}u(t)\dif\mu_{g}\geq \ln\left(\dfrac12\int_{\Sigma}e^{u_0}\dif\mu_{g}\right)\abs{A(t)}_{g}\geq-C,
\end{align*}
and
\begin{align*}
\int_{A(t)}u(t)\dif\mu_{g}\leq\int_{A(t)}e^{u(t)}\leq\int_{\Sigma}e^{u(t)}=\int_{\Sigma}e^{u_0}\leq C.
\end{align*}
Thus
\begin{align*}
\abs{\int_{A(t)}u(t)\dif\mu_{g}}\leq C.
\end{align*}
Notice that
\begin{align*}
\int_{\Sigma}e^{u_0}\dif\mu_{g}=&\int_{\Sigma}e^{u(t)}\dif\mu_{g}\\
=&\int_{\Sigma\setminus A(t)}e^{u(t)}\dif\mu_{g}+\int_{A(t)}e^{u(t)}\dif\mu_{g}\\
\leq&\dfrac12\int_{\Sigma}e^{u_0}\dif\mu_{g}+C\abs{A(t)}_{g}\\
=&\dfrac12\int_{\Sigma}e^{u(t)}\dif\mu_{g}+C\abs{A(t)}_{g},
\end{align*}
we get
\begin{align*}
\abs{A(t)}_{g}\geq C^{-1}.
\end{align*}
By Poincar\'e inequality, 
\begin{align*}
\norm{u(t)}_{L^2\left(\Sigma\right)}\leq& C\norm{\nabla_{g}u(t)}_{L^2\left(\Sigma\right)}+\abs{\bar{u}(t)}\\
\leq&C\norm{\nabla_{g}u(t)}_{L^2\left(\Sigma\right)}+\abs{\int_{A(t)}u(t)\dif\mu_{g}}+\abs{\int_{\Sigma\setminus A(t)}u(t)\dif\mu_{g}}\\
\leq&C\norm{\nabla_{g}u(t)}_{L^2\left(\Sigma\right)}+C+\sqrt{\abs{\Sigma\setminus A(t)}_{g}}\norm{u(t)}_{L^2\left(\Sigma\right)}.
\end{align*}
Hence
\begin{align*}
\norm{u(t)}_{L^2\left(\Sigma\right)}\leq& C\norm{\nabla_{g}u(t)}_{L^2\left(\Sigma\right)}+C.
\end{align*}
Notice that
\begin{align*}
\dfrac12\int_{\Sigma}\abs{\nabla_{g}u(t)}_{g}^2\dif\mu_{g}= J(u(t))-8\pi\bar{u}(t)+8\pi\ln\int_{\Sigma}he^{u(t)}\dif\mu_{g}\leq C+C\bar{u}(t).
\end{align*}
By Young's inequality, we conclude that
\begin{align*}
\norm{u(t)}_{H^1\left(\Sigma\right)}\leq C.
\end{align*}

\item[Step 3]
\begin{quotation}
    $\lim_{t\to\infty}\int_{\Sigma}e^{u(t)}\abs{\frac{\partial u(t)}{\partial t}}^2\dif\mu_{g}=0$.
\end{quotation}

We will follow the arguments of Brendle \cite{Bre03global} (see also \cite{casteras2015mean}). For every $\varepsilon>0$, there exist $k_0$ such  that for all $k\geq k_0$ 
\begin{align*}
\int_{\Sigma}e^{u(t_{n_k})}\abs{\dfrac{\partial u(t_{n_k})}{\partial t}}^2\dif\mu_{g}<\varepsilon.
\end{align*}
Assume for all $k\geq k_0$, 
\begin{align*}
m_k=\inf\set{t>t_{n_{k}}: \int_{\Sigma}e^{u(t)}\abs{\dfrac{\partial u(t)}{\partial t}}^2\dif\mu_{g}\geq2\varepsilon}<\infty.
\end{align*}
For $t_{n_k}\leq t\leq m_k$, we have
\begin{align*}
\int_{\Sigma}e^{u(t)}\abs{\dfrac{\partial u(t)}{\partial t}}^2\dif\mu_{g}\leq 2\varepsilon.
\end{align*}
Since $u(t)$ is bounded in $H^1\left(\Sigma\right)$ and $u(t)^{+}$ is bounded in $L^{\infty}\left(\Sigma\right)$, we conclude that
\begin{align*}
\abs{\Delta u(t)}\leq C\varepsilon+C,\quad\forall t_{n_k}\leq t\leq t_{m_k}.
\end{align*}
Thus
\begin{align*}
\norm{u(t)}_{L^{\infty}\left(\Sigma\right)}\leq C\norm{u(t)}_{H^2\left(\Sigma\right)}\leq C_{\varepsilon},\quad\forall t_{n_k}\leq t\leq m_k.
\end{align*}
Set
\begin{align*}
y(t)=\int_{\Sigma}e^{u(t)}\abs{\dfrac{\partial u(t)}{\partial t}}^2\dif\mu_{g}.
\end{align*}
Denote by $\dot u =\frac{\partial u}{\partial t},\ \ddot u=\frac{\partial^2u}{\partial t^2}$. Notice that
\begin{align*}
\dot u=e^{-u}\left(\Delta_{g}u-8\pi\right)+\dfrac{8\pi h}{\int_{\Sigma}he^{u}\dif\mu_{g}}.
\end{align*}
We get
\begin{align*}
\ddot u=&e^{-u}\Delta_{g}\dot u-\dot ue^{-u}\left(\Delta_{g}u-8\pi\right)-\dfrac{8\pi h\int_{\Sigma}he^{u}\dot u\dif\mu_{g}}{\left(\int_{\Sigma}he^{u}\dif\mu_{g}\right)^2}\\
=&e^{-u}\Delta_{g}\dot u-\dot u^2+\dfrac{8\pi h\dot u}{\int_{\Sigma}he^{u}\dif\mu_{g}}-\dfrac{8\pi h\int_{\Sigma}he^{u}\dot u\dif\mu_{g}}{\left(\int_{\Sigma}he^{u}\dif\mu_{g}\right)^2}.
\end{align*}
Hence
\begin{align*}
\dot y=&\int_{\Sigma}\left(e^{u}\dot u^3+2e^{u}\dot u\ddot u\right)\dif\mu_{g}\\
=&-2\int_{\Sigma}\abs{\nabla_{g}\dot u}^2\dif\mu_{g}-\int_{\Sigma}e^{u}\dot u^3\dif\mu_{g}+16\pi\left[\dfrac{\int_{\Sigma} he^{u}\dot u^2\dif\mu_{g}}{\int_{\Sigma}he^{u}\dif\mu_{g}}-\left(\dfrac{\int_{\Sigma}he^{u}\dot u\dif\mu_{g}}{\int_{\Sigma}he^{u}\dif\mu_{g}}\right)^2\right]\\
\leq&-2\int_{\Sigma}\abs{\nabla_{g}\dot u}^2\dif\mu_{g}-\int_{\Sigma}e^{u}\dot u^3\dif\mu_{g}+Cy.
\end{align*}
We estimate the second term in the RHS of the above inequality as follows: for all $t_{n_k}\leq t\leq m_k$,
\begin{align*}
-\int_{\Sigma}e^{u}\dot u^3\dif\mu_{g}\leq&C\int_{\Sigma}\abs{\dot u}^3\dif\mu_{g}\\
\leq&C\norm{\dot u}_{L^2\left(\Sigma\right)}^2\norm{\dot u}_{H^1\left(\Sigma\right)}\\
\leq&C_{\varepsilon}y\left(y+\int_{\Sigma}\abs{\nabla_{g}\dot u}^2_{g}\dif\mu_{g}\right)^{1/2}.
\end{align*}
Since $\int_{\Sigma}e^{u}\dot u\dif\mu_{g}=0$, applying the Poincar\'e inequality to obtain
\begin{align*}
\int_{\Sigma}e^{u}\dot u^2\dif\mu_{g}\leq\dfrac{1}{\lambda_{1,e^{u}g}}\int_{\Sigma}\abs{\nabla_{e^{u}g}\dot u}^2_{e^{u}g}\dif\mu_{e^{u}g}=\dfrac{1}{\lambda_{1,e^{u}g}}\int_{\Sigma}\abs{\nabla_{g}\dot u}^2_{g}\dif\mu_{g}\leq C\int_{\Sigma}\abs{\nabla_{g}\dot u}^2_{g}\dif\mu_{g}.
\end{align*}
Thus for all $t_{n_k}\leq t\leq m_k$,
\begin{align*}
-\int_{\Sigma}e^{u}\dot u^3\dif\mu_{g}\leq C_{\varepsilon}y^{1/2}\left(\int_{\Sigma}\abs{\nabla_{g}\dot u}^2_{g}\dif\mu_{g}\right)^{1/2},
\end{align*}
which implies
\begin{align*}
\dot y\leq C_{\varepsilon}y.
\end{align*}
Hence
\begin{align*}
y(t_{m_k})\leq y(t_{n_k})+C_{\varepsilon}\int_{t_{n_k}}^{\infty}y(t)\dif t.
\end{align*}
Thus
\begin{align*}
\varepsilon\leq C_{\varepsilon}\int_{t_{n_k}}^{\infty}y(t)\dif t\to0,\quad\text{as}\quad  t_{n_k}\to\infty
\end{align*}
which is a contradiction. Therefore
\begin{align*}
\lim_{t\to\infty}\int_{\Sigma}e^{u(t)}\abs{\dfrac{\partial u(t)}{\partial t}}^2\dif\mu_{g}=0.
\end{align*}

\item[Step 4] 
\begin{quotation}
    $
\norm{u(t)}_{H^2\left(\Sigma\right)}\leq C$ which implies that $\norm{u(t)}_{C^{\gamma}\left(\Sigma\right)}\leq C_{\gamma}$ for every $0<\gamma<1$.
\end{quotation}

This is a direct consequence of the standard elliptic estimates and Sobolev inequalities.

\item[Step 5]
\begin{quotation}
    $\lim_{t\to\infty}\norm{u(t)-u_\infty}_{L^2\left(\Sigma\right)}=0$ implies  $\lim_{t\to\infty}\norm{u(t)-u_{\infty}}_{H^2\left(\Sigma\right)}=0$.
\end{quotation}

Since
\begin{align*}
\dfrac{\partial e^{u}}{\partial t}=\Delta_{g}\left(u-u_{\infty}\right)+8\pi\left(\dfrac{he^{u}}{\int_{\Sigma}he^{u}\dif\mu_{g}}-\dfrac{he^{u_{\infty}}}{\int_{\Sigma}he^{u_{\infty}}\dif\mu_{g}}\right),
\end{align*}
we get
\begin{align*}
\abs{\Delta_{g}\left(u(t)-u_{\infty}\right)}\leq C\left(\abs{\dfrac{\partial u(t)}{\partial t}}+\abs{u(t)-u_{\infty}}+\norm{u(t)-u_{\infty}}_{L^1\left(\Sigma\right)}\right)
\end{align*}
which implies 
\begin{align*}
\norm{u(t)-u_{\infty}}_{H^2\left(\Sigma\right)}\leq C\left(\norm{\dfrac{\partial u(t)}{\partial t}}_{L^2\left(\Sigma\right)}+\norm{u(t)-u_{\infty}}_{L^2\left(\Sigma\right)}\right).
\end{align*}
The claim follows by letting $t\to+\infty$.

\item[Step 6]
\begin{quotation}
    There are positive constants $\sigma$ and $\theta\in(1/2,1)$ such that
\begin{align*}
\forall u\in H^2\left(\Sigma\right),\ \norm{u-u_{\infty}}_{L^2\left(\Sigma\right)}<\sigma \quad \Longrightarrow \quad \abs{J(u)-J(u_{\infty})}^{\theta}\leq \norm{\mathcal{M}(u)}_{H^{2}\left(\Sigma\right)}.
\end{align*}
\end{quotation}

Notice that the functional $J:H^1\left(\Sigma\right)\To\mathbb{R}$ is analytic and the gradient map $\mathcal{M}:H^1\left(\Sigma\right)\To H^{-1}\left(\Sigma\right)$ is given by
\begin{align*}
u\mapsto\mathcal{M}(u)=-\Delta_{g}u-8\pi\left(\dfrac{he^u}{\int_{\Sigma}he^{u}\dif\mu_{g}}-1\right).
\end{align*}
The Jacobi operator $\mathcal{L}:H^1\left(\Sigma\right)\To H^{-1}\left(\Sigma\right)$ of $J$ at a critical point $u\in C^{\infty}\left(\Sigma\right)$ of $J$ is given by
\begin{align*}
\xi\mapsto\mathcal{L}(\xi)=-\Delta_{g}\xi-8\pi\left(\dfrac{he^{u}\xi}{\int_{\Sigma}he^{u}\dif\mu_{g}}-\dfrac{he^{u}\int_{\Sigma}he^{u}\xi\dif\mu_{g}}{\left(\int_{\Sigma}he^{u}\dif\mu_{g}\right)^2}\right)
\end{align*}
is a Fredohom operator with index zero. 
Since $\mathcal{M}\left(H^2\left(\Sigma\right)\right)\subset L^2\left(\Sigma\right)$, applying the {\L}ojasiewicz-Simon gradient inequality (cf. \cite[Proposition 1.3]{Jen98simple} or \cite[Theorem 2]{FeeMar19Lojasiewicz}),  there are positive constants $\tilde\sigma$ and $\theta\in(1/2,1)$ such that
\begin{align*}
\forall u\in H^2\left(\Sigma\right),\ \norm{u-u_{\infty}}_{H^2\left(\Sigma\right)}<\tilde\sigma \quad \Longrightarrow \quad \abs{J(u)-J(u_{\infty})}^{\theta}\leq \norm{\mathcal{M}(u)}_{L^{2}\left(\Sigma\right)}.
\end{align*}
Hence we obtain this claim by choosing $\sigma$ small.

\item[Step 7]
\begin{quotation}
    $\lim_{t\to\infty}\norm{u(t)-u_{\infty}}_{L^2\left(\Sigma\right)}=0$ which gives the global convergence.
\end{quotation}

We will flow the approach of Jendoubi \cite{Jen98simple}. For every $0<\varepsilon<<\sigma$, there exist $k_1$ such that for all $k\geq k_1$,
\begin{align*}
\norm{u(t_{k})-u_{\infty}}_{L^2\left(\Sigma\right)}<\varepsilon.
\end{align*}
Assume for all $k\geq k_1$,
\begin{align*}
s_k=\inf\set{t>t_{n_k}: \norm{u(t)-u_{\infty}}_{L^2\left(\Sigma\right)}\geq \sigma}<\infty.
\end{align*}
Then for all $n_k\leq t<s_k$,
\begin{align*}
\norm{u(t)-u_{\infty}}_{L^2\left(\Sigma\right)}<\sigma=\norm{u(s_k)-u_{\infty}}_{L^2\left(\Sigma\right)}.
\end{align*}
Without loss of generality, assume $J(u(t))>J(u_{\infty})$ for all $t>0$. For $t_{n_k}\leq t<s_k$, we have
\begin{align*}
-\dfrac{\dif}{\dif t}\left(J(u(t))-J(u_{\infty})\right)^{1-\theta}=&-\left(1-\theta\right)\left(J(u(t))-J(u_{\infty})\right)^{-\theta}\dfrac{\dif}{\dif t}J(u(t))\\
=&\left(1-\theta\right)\left(J(u(t))-J(u_{\infty})\right)^{-\theta}\norm{e^{u(t)/2}\dfrac{\partial u(t)}{\partial t}}_{L^2\left(\Sigma\right)}^2\\
\geq&(1-\theta)\norm{\dfrac{\partial u(t)}{\partial t}}_{L^2\left(\Sigma\right)}.
\end{align*}
Thus
\begin{align*}
\int_{t_{n_k}}^{s_k}\norm{\dfrac{\partial u(t)}{\partial t}}_{L^2\left(\Sigma\right)}\dif t\leq\dfrac{1}{1-\theta}\left(J(u(t_{n_k}))-J(u_{\infty})\right)^{1-\theta}.
\end{align*}
Since
\begin{align*}
\dfrac{\dif}{\dif t}\norm{u(t)-u_{\infty}}_{L^2\left(\Sigma\right)}\leq\norm{\dfrac{\partial u(t)}{\partial t}}_{L^2\left(\Sigma\right)},
\end{align*}
we get
\begin{align*}
\sigma=&\norm{u(s_k)-u_{\infty}}_{L^2\left(\Sigma\right)}\\
\leq&\norm{u(t_{n_k})-u_{\infty}}_{L^2\left(\Sigma\right)}+\int_{t_{n_k}}^{s_k}\norm{\dfrac{\partial u(t)}{\partial t}}_{L^2\left(\Sigma\right)}\dif t\\
\leq&\norm{\dfrac{\partial u(t_{n_k})}{\partial t}}_{L^2\left(\Sigma\right)}+\dfrac{1}{1-\theta}\left(J(u(t_{n_k}))-J(u_{\infty})\right)^{1-\theta}
\end{align*}
which is a contradiction when $n_{k}\to+\infty$. Hence we have $s_{k_2}=+\infty$ for some $k_2$. We conclude that
\begin{align*}
\int_{0}^{\infty}\norm{\dfrac{\partial u(t)}{\partial t}}_{L^2\left(\Sigma\right)}\dif t<+\infty
\end{align*}
which gives
\begin{align*}
\lim_{t\to\infty}\norm{u(t)-u_{\infty}}_{L^2\left(\Sigma\right)}=0.
\end{align*}

\end{description}

\end{proof}

%bibtex
%\biboptions{longnamesfirst,sort&compress}
%\bibliographystyle{elsarticle-harv}%abbrv,alpha,amsplain,amsalpha,elsarticle-harv,elsarticle-num,elsarticle-num-names

%\bibliography{sz}

\end{document}